\documentclass[a4paper, 11pt, pdftex]{amsart}

\usepackage{amsmath}
\usepackage{amssymb}
\usepackage{amsthm}
\usepackage[abbrev,alphabetic]{amsrefs}
\usepackage{amscd}
\usepackage{graphicx}
\usepackage{graphics}
\usepackage{ulem}
\usepackage{url}
\usepackage{ascmac}

\newcommand{\doi}[1]{\url{https://doi.org/#1}}

\usepackage[top=30truemm,bottom=25truemm,left=30truemm,right=30truemm]{geometry}

\usepackage[all,cmtip]{xy}

\usepackage{color}
\usepackage{xypic}

\title[Rationality of the growth series of virtually abelian groups]
{Rationality of the multivariate growth series for algebraic sets of virtually abelian groups}

\author{Yusuke Nakamura}
\address{Graduate School of Mathematics, Nagoya University, Furo-cho, Chikusa-ku, Nagoya, 464-8602, Japan.}
\email{y.nakamura@math.nagoya-u.ac.jp}
\urladdr{https://sites.google.com/site/ynakamuraagmath/}

\newtheorem{thm}{Theorem}[section]
\newtheorem{lem}[thm]{Lemma}

\newtheorem{prop}[thm]{Proposition}

\theoremstyle{definition}
\newtheorem{defi}[thm]{Definition}

\newtheorem{nota}[thm]{Notation}

\theoremstyle{remark}
\newtheorem{rmk}[thm]{Remark}

\newtheorem*{ackn}{Acknowledgements}

\begin{document}
%05A15 Exact enumeration problems, generating functions
%20F65 - Geometric group theory
%05E16 - Combinatorial aspects of groups and algebras

\begin{abstract}
We prove the rationality of the multivariate relative growth series for algebraic sets of virtually abelian groups, which had been conjectured by Evetts and Levine. 
\end{abstract}

\maketitle

\section{Introduction}

Let $G$ be a group, $S \subset G$ be a finite subset, and 
$\omega: S \to \mathbb{Z}_{>0}$ be a positive integer weight. 
Let $S^*$ be the set of all words from $S$. 
For a word $\sigma = s_1 \cdots s_{\ell} \in S^*$, we define $\omega(\sigma) := \sum_{i=1} ^{\ell} \omega(s_i)$. 
For $g \in G$, we define $\omega(g)$ as the smallest $\omega(\sigma)$ for all words $\sigma$ that represent $g$. 
The \textit{growth series} $\mathbb{S}_{G, S,\omega} (t)$ of $G$ with respect to $S$ and $\omega$ is defined as a formal power series
\[
\mathbb{S}_{G, S, \omega}(t) := \sum _{i \ge 0} f_{G, S, \omega}(i) t^i \in \mathbb{Z}[[t]], 
\]
where $f_{G,S,\omega}(i) := \# \{ g \in G \mid \omega(g) = i \}$. 
In \cite{Ben83}, Benson proved that the growth series $\mathbb{S}_{G, S, \omega}(t)$ is a rational function when $G$ is a virtually abelian group. 

An \textit{$n$-dimensional periodic graph} $(\Gamma, L)$ is a pair of a directed graph $\Gamma$ (that may have loops and multiple edges) and a free abelian group $L$ of rank $n$ such that $L$ freely acts on $\Gamma$ and its quotient graph $\Gamma/L$ is finite (see Definition \ref{defi:pg}). 
For a vertex $x_0$ of $\Gamma$, 
the \textit{growth series} $\mathbb{S}_{\Gamma, x_0} (t)$ of $\Gamma$ with respect to $x_0$ is defined by 
\[
\mathbb{S}_{\Gamma, x_0}(t) := \sum _{i \ge 0} f_{\Gamma, x_0}(i) t^i \in \mathbb{Z}[[t]], 
\]
where $f_{\Gamma, x_0}(i)$ is defined as the number of vertices of $\Gamma$ whose distance from $x_0$ is $i$. 
The sequence $(f_{\Gamma, x_0}(i))_i$ is studied in both areas of crystallography and combinatorics (cf.\ \cite{GS19}*{Introduction}). 
Furthermore, this is also studied in relation to the Ehrhart theory (cf.\ \cite{IN1, IN2}). 

In \cite{GKBS96}, Grosse-Kunstleve, Brunner and Sloane conjectured that the growth series $\mathbb{S}_{\Gamma, x_0} (t)$ of periodic graphs $\Gamma$ are always rational functions. 
In \cite{NSMN21}, the author, Sakamoto, Mase and Nakagawa proved that this conjecture is true (see Theorem \ref{thm:NSMN}). 
This result is a generalization of Benson's theorem since the Cayley graphs of virtually abelian groups are periodic graphs (see Remark \ref{rmk:benson}). 
In the proof of Theorem \ref{thm:NSMN}, they use the commutative monoid theory, and the proof is essentially different from Benson's original proof of Theorem \ref{thm:Benson}, where he uses his theory of ``polyhedral sets". 

The rationality of the growth series $\mathbb{S}_{G, S, \omega}(t)$ is also proved for hyperbolic groups in \cites{Can84, Gro87} and for the integer Heisenberg group in \cite{DS19}. 
Furthermore, some variants of growth series, such as the conjugacy growth series and the relative growth series, have been studied for many groups (cf.\ \cites{GS10, DO15, AC17, CHHR16, Eve19}). 

In \cite{EL22}, Evetts and Levine generalize Benson's theorem to the univariate relative growth series for algebraic sets of virtually abelian groups. 
For a positive integer $d$, a subset $U \subset G^d$ is called \textit{algebraic} if $U$ is the set of solutions of a finite system of equations with $d$ variables whose coefficients are in $G$ (see Definition \ref{defi:as}). 
The \textit{univariate (relative) growth series} $\mathbb{S}^{\rm uni}_{U, S, \omega}(t)$ of $U$ with respect to $S$ and $\omega$ is defined by
\[
\mathbb{S}^{\rm uni}_{U, S, \omega}(t) := \sum _{{\bf x} \in U} t^{\omega ({\bf x})} \in \mathbb{Z}[[t]], 
\]
where $\omega ({\bf x}) := \sum _{i = 1} ^{\ell} \omega (x_i)$. 
Furthermore, the \textit{multivariate (relative) growth series} $\mathbb{S}_{U, S, \omega}({\bf z})$ is defined by
\[
\mathbb{S}_{U, S, \omega}({\bf z}) := \sum _{{\bf x} \in U} z_1^{\omega(x_1)} z_2^{\omega(x_2)} \cdots z_d^{\omega(x_d)} \in \mathbb{Z}[[z_1, \ldots, z_d]]. 
\]
In \cite{EL22}, for any algebraic sets of any virtually abelian groups, the univariate growth series are proved to be rational functions (\cite{EL22}*{Theorem 4.3}), and furthermore, the multivariate growth series are proved to be holonomic functions (\cite{EL22}*{Corollary 4.21}). 
However, the rationality of the multivariate growth series remained a conjecture (\cite{EL22}*{Conjecture 5.1}). 
The purpose of this paper is to give an affirmative answer to this conjecture. 
\begin{thm}[$=$ Theorem \ref{thm:main}]\label{thm:main_i}
Let $G$ be a finitely generated virtually abelian group, $S \subset G$ a finite subset, and $\omega$ a positive integer weight. 
Then, for any algebraic subset $U \subset G^d$ of $G$, its multivariate growth series $\mathbb{S}_{Y, S, \omega}({\bf z})$ is a rational function. 
\end{thm}

Theorem \ref{thm:main_i} is a consequence of Theorem \ref{thm:main_p}, which is the corresponding theorem for periodic graphs. 
The idea of the proof of Theorem \ref{thm:main_p} is to use an algebraic structure on the \textit{graded growth set} 
\[
B := \{ (i, y) \in \mathbb{Z}_{\ge 0} \times V_{\Gamma} \mid d_{\Gamma}(x_0, y) \le i \} \subset \mathbb{Z}_{\ge 0} \times V_{\Gamma}, 
\]
which is studied in \cite{NSMN21} (see Appendix \ref{section:NSMN}). 
In \cite{IN2}, further structure of $B$, which is not used in this paper, is studied as an analogy to the Ehrhart ring.
We believe that these ideas have potential for further applications to similar studies of virtually abelian groups.

The paper is organized as follows. 
In Section \ref{section:pre}, we introduce some notations on graphs and groups, and we also summarize facts on commutative monoid theory which will be used in Section \ref{section:proof}. 
In Section \ref{section:proof}, we prove Theorem \ref{thm:main_p}. 
In Section \ref{section:proof2}, we prove Theorem \ref{thm:main}.
In Appendix \ref{section:NSMN}, we explain the algebraic structure of the graded growth set $B$ proved in \cite{NSMN21}. 
The proofs in Appendix \ref{section:NSMN} are essentially the same as those in \cite{NSMN21}, but streamlined for mathematicians.

\begin{ackn}
The author would like to thank Professor Takeo Uramoto for pointing him to the connection between the paper \cite{NSMN21} and the geometric group theory of virtually abelian groups.
He would also like to thank Takuya Inoue for many discussions. 
The author is partially supported by JSPS KAKENHI No.\ 18K13384 and 22K13888. 
\end{ackn}

\section{Preliminaries}\label{section:pre}

\begin{nota}\hfill
\begin{itemize}
\item For a set $X$, $\# X $ denotes its cardinality, and $2^{X}$ denotes its power set. 
\end{itemize}
\end{nota}

\subsection{Growth series of periodic graphs}\label{subsection:pg}
In this paper, a \textit{graph} means a directed weighted graph which may have loops and multiple edges. 
More precisely, a graph $\Gamma = (V_{\Gamma}, E_{\Gamma}, s_{\Gamma}, t_{\Gamma}, w_{\Gamma})$ 
consists of the set $V_{\Gamma}$ of vertices, the set $E_{\Gamma}$ of edges, 
the source function $s_{\Gamma}: E_{\Gamma} \to V_{\Gamma}$, 
the target function $t_{\Gamma}: E_{\Gamma} \to V_{\Gamma}$, and 
the weight function $w_{\Gamma}: E_{\Gamma} \to \mathbb{Z}_{> 0}$. 
We abbreviate $s_{\Gamma}$, $t_{\Gamma}$ and $w_{\Gamma}$ to $s$, $t$ and $w$ 
when no confusion can arise. 

Let $\Gamma$ be a graph. 
A \textit{walk} $p$ is a sequence $e_1 e_2 \cdots e_{\ell}$ of edges $e_i \in V_{\Gamma}$ satisfying $t(e_i) = s(e_{i+1})$ for all $1 \le i \le \ell -1$. 
We define the \textit{weight} $w(p)$ of $p$ by $w(p) := \sum _{i = 1} ^{\ell} w(e_i)$. 
For $x, y \in V_{\Gamma}$, we define $d_{\Gamma} (x,y) \in \mathbb{Z}_{\ge 0} \cup \{ \infty \}$ as the smallest weight $w(p)$ of any walk $p = e_1 e_2 \cdots e_{\ell}$ with $x = s(e_1)$ and $y = t(e_{\ell})$. 
Then, for $x_0 \in V_{\Gamma}$ and $i \in \mathbb{Z}_{\ge 0}$, 
$f_{\Gamma, x_0}(i) \in \mathbb{Z}_{\ge 0}$ is defined as
\[
f_{\Gamma, x_0}(i) := \# \{ y \in V_{\Gamma} \mid d_{\Gamma}(x_0, y) = i \}. 
\]
The sequence $(f_{\Gamma, x_0}(i))_i$ is called the \textit{growth sequence} of $\Gamma$ with respect to $x_0$. 
The \textit{growth series} $\mathbb{S}_{\Gamma, x_0} (t)$ of $\Gamma$ with respect to $x_0$ is defined as its generating function 
\[
\mathbb{S}_{\Gamma, x_0}(t) := \sum _{i \ge 0} f_{\Gamma, x_0}(i) t^i. 
\]

\begin{defi}[cf.\ \cite{IN1}]\label{defi:pg}
\begin{enumerate}
\item
Let $\Gamma$ be a graph, and let $G$ be a group. 
A \textit{$G$-action} on $\Gamma$ means $G$-actions on $V_\Gamma$ and $E_\Gamma$ which preserve the edge relations and the weight function, i.e.,\ for any $u \in G$ and $e \in E_\Gamma$, we have 
\[
s_\Gamma(u(e))=u(s_\Gamma(e)), \quad t_\Gamma(u(e))=u(t_\Gamma(e)), \quad w_{\Gamma}(u(e))=w_{\Gamma}(e).
\]
\item 
Let $n$ be a positive integer. 
An \textit{$n$-dimensional periodic graph} $(\Gamma, L)$ is 
a free abelian group $L \simeq \mathbb{Z}^n$ of rank $n$ and a graph $\Gamma$ with an $L$-action satisfying the following two conditions:
\begin{itemize}
\item 
The $L$-actions on $V_\Gamma$ and $E_\Gamma$ are free, i.e., no non-trivial element of $G$ fixes an element of $V_{\Gamma}$ and $E_{\Gamma}$. 
\item
Their quotient sets $V_\Gamma / L$ and $E_\Gamma / L$ are finite sets.
\end{itemize}
\end{enumerate}
\end{defi}

The growth series of periodic graphs are known to be rational functions. 

\begin{thm}[\cite{NSMN21}*{Theorem 2.2}]\label{thm:NSMN}
Let $(\Gamma, L)$ be a periodic graph, and let $x_0 \in V_{\Gamma}$. 
Then, its growth sequence $(f_{\Gamma, x_0}(i))_i$ is quasi-polynomial type. 
In particular, its growth series $\mathbb{S}_{\Gamma, x_0}(t)$ is a rational function. 
\end{thm}

\subsection{Growth series of groups}\label{subsection:gsg}
Let $G$ be a group, $S \subset G$ a finite subset, and 
$\omega: S \to \mathbb{Z}_{>0}$ a positive integer weight. 
Let $S^*$ be the set of all words from $S$. 
For a word $\sigma = s_1 \cdots s_{\ell} \in S^*$, we define 
$\omega(\sigma) := \sum_{i=1} ^{\ell} \omega(s_i)$. 
For $g \in G$, we define $\omega(g)$ as the smallest $\omega(\sigma)$ 
for all words $\sigma$ that represent $g$. 
By convention, we define $\omega(g) = \infty$ if $g$ cannot be represented by $S$. 
Then, for $i \in \mathbb{Z}_{\ge 0}$, 
$f_{G,S,\omega}(i) \in \mathbb{Z}_{\ge 0}$ is defined as
\[
f_{G,S,\omega}(i) := \# \{ g \in G \mid \omega(g) = i \}. 
\]
The sequence $(f_{G,S,\omega}(i))_i$ is called the \textit{growth sequence} of $G$ with respect to $S$ and $\omega$. 
The \textit{growth series} $\mathbb{S}_{\Gamma, S,\omega} (t)$ of $G$ with respect to $S$ and $\omega$ is its generating function 
\[
\mathbb{S}_{G, S, \omega}(t) := \sum _{i \ge 0} f_{G, S, \omega}(i) t^i. 
\]

\begin{rmk}
We need not to assume that $S$ generates $G$ even in Theorems \ref{thm:Benson} and \ref{thm:main}. 
\end{rmk}

\begin{rmk}
The (right) \textit{Cayley graph} $\Gamma = \Gamma _{G, S, \omega}$ of $G$ with respect to $S$ and $\omega$ is defined as follows:
\begin{itemize}
\item 
$V_{\Gamma} := G$. 

\item 
$E_{\Gamma} := G \times S$. 

\item 
For $e = (g, s) \in E_{\Gamma}$, we define 
\[
s_{\Gamma}(e) = g, \qquad 
t_{\Gamma}(e) = gs, \qquad
w_{\Gamma}(e) = \omega(s). 
\]
\end{itemize}
Then, the number $f_{G,S,\omega}(i)$ defined above 
coincides with the growth sequence $f_{\Gamma, 1_G}(i)$ of the graph $\Gamma$ defined in Section \ref{subsection:pg}. 

\end{rmk}

A group $G$ is called \textit{virtually abelian} if there exists an abelian subgroup $H$ of $G$ with finite index $[G:H] < \infty $. 
In \cite{Ben83}, Benson proved that the growth series of virtually abelian groups are rational functions. 

\begin{thm}[{Benson \cite{Ben83}}]\label{thm:Benson}
Let $G$ be a finitely generated virtually abelian group, $S \subset G$ a finite subset, and $\omega$ a positive integer weight. 
Then, the sequence $(f_{G, S, \omega}(i))_i$ is of quasi-polynomial type. 
In particular, its growth series $\mathbb{S}_{G, S, \omega}(t)$ is a rational function. 
\end{thm}

\begin{rmk}\label{rmk:benson}
When $G$ is a finitely generated virtually abelian group, we may take a free abelian subgroup $L < G$ of finite rank and of finite index. 
Since $L$ is a subgroup of $G$, the left $L$-action on $G$ induces a free $L$-action on the Cayley graph $\Gamma = \Gamma _{G,S,\omega}$. Furthermore, since $L$ is a subgroup of $G$ with finite index,  the quotient graph $\Gamma /L$ is a finite graph.
Therefore, the Cayley graph $\Gamma$ is a periodic graph. 
Hence, Theorem \ref{thm:Benson} can be seen as a special case of Theorem \ref{thm:NSMN}. 

Benson's original proof of Theorem \ref{thm:Benson} uses his theory of ``polyhedral sets". 
On the other hand, the proof of Theorem \ref{thm:NSMN} uses the theory of commutative monoid explained in Subsection \ref{subsection:monoid}, and the argument is quite different from Benson's argument. 
\end{rmk}

\subsection{Some facts in commutative monoid theory}\label{subsection:monoid}
The key tool of this paper is the rationality of the Hilbert series of finitely generated modules (Theorem \ref{thm:Serre}). 
This is an important theorem in the context of algebraic geometry and commutative algebra, but it is also useful in enumerative combinatorics.

A commutative monoid $M$ is called \textit{integral} if it has the cancellation property (i.e., for any $a,b,c \in M$, $a + b = a+c$ implies $b=c$). 
For a commutative monoid $M$, an \textit{$M$-module} just means a set with an $M$-action. 
An $M$-module $X$ is called \textit{finitely generated} if $X = \bigcup _{x \in X'} (M+x)$ holds for some finite subset $X' \subset X$. 
An $M$-module $X$ is called a \textit{free $M$-module} if there exists a subset $X' \subset X$ such that any $x \in X$ can be uniquely expressed as $x = a + y$ with $a \in M$ and $y \in X'$. 
For more detail on commutative monoid and its module theory, refer to \cite{BG09} and \cite{NSMN21}. 

Theorem \ref{thm:fg}(1) below is well-known in the theory of commutative monoid (cf.\ \cite{Ogu}*{Ch.\ I, Theorem 2.1.17.6}, \cite{BG09}*{Corollary 2.11}). 
Theorem \ref{thm:fg}(2) is proved in \cite{NSMN21}*{Theorem A12}. 
We note that Theorem \ref{thm:fg}(2) can be seen as a discrete version of Motzkin's theorem (\cite{BG09}*{Theorem 1.27}). 
\begin{thm}\label{thm:fg}
Let $M$ be an integral commutative monoid, and let $X$ be a free $M$-module. 
Let $N_1$ and $N_2$ be finitely generated submonoids of $M$. 
Let $Y_i \subset X$ be a finitely generated $N_i$-submodule for $i=1,2$. 
Then the following assertions hold:
\begin{enumerate}
\item 
$N_1 \cap N_2$ is a finitely generated monoid. 
\item 
$Y_1 \cap Y_2$ is a finitely generated $(N_1 \cap N_2)$-module. 
\end{enumerate}
\end{thm}

The following theorem due to Hilbert and Serre is well-known in the theory of commutative algebra. 
The series $H_Y$ is called the \textit{Hilbert series}.   
\begin{thm}[{cf.\ \cite{BG09}*{Theorem 6.37}}]\label{thm:Serre}
Let $d$ be a positive integer. 
Let $M$ be a commutative monoid, and let $X$ be an $M$-module. 
Let $N' \subset \mathbb{Z}_{> 0} ^d \times M$ be a finite subset, and 
let $N \subset \mathbb{Z}_{\ge 0} ^d \times M$ be the submonoid generated by $N'$. 
Let $Y \subset \mathbb{Z}_{\ge 0}^d \times X$ be a finitely generated $N$-submodule. 
We define a function $h_Y$ by 
\[
h_Y: \mathbb{Z}_{\ge 0}^d \to \mathbb{Z}_{\ge 0}; \quad  {\bf a} \mapsto \# \{ x \in X \mid ({\bf a},x) \in Y \}
\]
Then, the series 
\[
H_Y ({\bf t}) := 
\sum _{{\bf a} \in \mathbb{Z}_{\ge 0}^d} h_Y({\bf a}) {\bf t}^{\bf a} \in \mathbb{Z}[[t_1, \ldots , t_d]]
\]
is a rational function. 
Here, ${\bf t}^{\bf a}$ denotes $t_1 ^{a_1} \cdots t_d ^{a_d}$ for ${\bf a} = (a_1, \ldots , a_d)$. 
\end{thm}

\section{Rationality of the growth series of periodic graphs}\label{section:proof}
This section is devoted to proving Theorem \ref{thm:main_p}. 
Throughout this section, we fix an $n$-dimensional periodic graph $(\Gamma, L)$. 
We also fix $x_0 \in V_{\Gamma}$. 

We define a \textit{graded growth set} $B$ of $\Gamma$ with respect to $x_0$ by 
\[
B := \{ (i, y) \in \mathbb{Z}_{\ge 0} \times V_{\Gamma} \mid d_{\Gamma}(x_0, y) \le i \} \subset \mathbb{Z}_{\ge 0} \times V_{\Gamma}. 
\]
Note that $V_{\Gamma}$ has a free $L$-module structure by the free $L$-action on $V_{\Gamma}$. 
Therefore, $\mathbb{Z}_{\ge 0} \times V_{\Gamma}$ has a free $\bigl( \mathbb{Z}_{\ge 0} \times L \bigr)$-module structure. 
The key point of the proof of Theorem \ref{thm:main_p} is to use the following algebraic structure of $B$. 
\begin{thm}[\cite{NSMN21}]\label{thm:NSMN_decomp}
For each subset $S \subset V_{\Gamma}/L$, there exist subsets $M_S \subset \mathbb{Z}_{\ge 0} \times L$ and 
$X_S \subset \mathbb{Z}_{\ge 0} \times V_{\Gamma}$ with the following conditions: 
\begin{enumerate}
\item 
We have $B = \bigcup _{S \in 2 ^{V_{\Gamma}/L}} X_S$. 

\item 
Each $M_S$ is a finitely generated submonoid of $\mathbb{Z}_{\ge 0} \times L$. 

\item 
Each $X_S$ is a finitely generated $M_S$-module. 
\end{enumerate}
\end{thm}

\noindent
In Appendix \ref{section:NSMN}, we give constructions of $M_S$ and $X_S$, and we also give a proof of Theorem \ref{thm:NSMN_decomp} for readers' convenience. 
In this section, we fix $M_S$ and $X_S$ that satisfy the properties (1)-(3).

\begin{defi}
Let $d$ be a positive integer. 
We identify the two sets $\bigl( \mathbb{Z}_{\ge 0} \times L \bigr) ^d$ and $\mathbb{Z}_{\ge 0}^d \times L^d$ 
(resp.\  $\bigl( \mathbb{Z}_{\ge 0} \times V_{\Gamma} \bigr) ^d$ and $\mathbb{Z}_{\ge 0}^d \times V_{\Gamma}^d$) by the map 
\[
\bigl( (a_1,b_1), \ldots, (a_d, b_d) \bigr) \mapsto \bigl( (a_1, \ldots, a_d), (b_1, \ldots , b_d) \bigr). 
\]

\begin{enumerate}
\item 
For a subset $Y \subset V_{\Gamma} ^d$, 
we define $B_Y \subset \mathbb{Z}_{\ge 0} ^d \times V_{\Gamma} ^d$ by 
\[
B_Y := B ^d \cap \bigl( \mathbb{Z}_{\ge 0} ^d \times Y \bigr). 
\]

\item
For ${\bf S} = (S_1, \ldots , S_d) \in \bigl( 2^{V_{\Gamma} / L} \bigr) ^d$, we define $M_{\bf S}$ and $X_{\bf S}$ by
\begin{align*}
M_{\bf S} &:= M_{S_1} \times \cdots \times M_{S_d} \subset \bigl( \mathbb{Z}_{\ge 0} \times L \bigr) ^d, \\
X_{\bf S} &:= X_{S_1} \times \cdots \times X_{S_d} \subset \bigl( \mathbb{Z}_{\ge 0} \times V_{\Gamma} \bigr) ^d. 
\end{align*}
For subsets $N \subset L^d$ and $Y \subset V_{\Gamma} ^d$, 
we also define 
\[
M_{N, \bf S} := M_{\bf S} \cap \bigl ( \mathbb{Z}_{\ge 0} ^d \times N \bigr), \qquad 
X_{Y, \bf S} := X_{\bf S} \cap \bigl ( \mathbb{Z}_{\ge 0} ^d \times Y \bigr).  
\]
\end{enumerate}
\end{defi}

In the following lemma, we assume that $N$ is a submonoid of $L^d$. 
Since $V_{\Gamma} ^d$ is an $L^d$-module, $V_{\Gamma} ^d$ naturally has an $N$-module structure. 

\begin{lem}\label{lem:fgm}
Let $d$ be a positive integer. 
Let $N$ be a finitely generated submonoid of $L^d$. 
Let $Y$ be a finitely generated $N$-submodule of $V_{\Gamma} ^d$. 
For any subset $\Lambda \subset \bigl( 2^{V_{\Gamma} / L} \bigr) ^d$, the following assertions hold:
\begin{enumerate}
\item 
$\bigcap _{{\bf S} \in \Lambda} M_{N, \bf S}$ is a finitely generated monoid. 

\item
$\bigcap _{{\bf S} \in \Lambda} X_{Y, \bf S}$ is a finitely generated $\bigl( \bigcap _{{\bf S} \in \Lambda} M_{N, \bf S} \bigr)$-module. 
\end{enumerate}
\end{lem}
\begin{proof}
For each ${\bf S} = (S_1, \ldots , S_d) \in \bigl( 2^{V_{\Gamma} / L} \bigr) ^d$, 
the product $M_{\bf S} = M_{S_1} \times \cdots \times M_{S_d}$ is a finitely generated monoid since $M_{S_1}, \ldots , M_{S_d}$ are finitely generated monoids by Theorem \ref{thm:NSMN_decomp}(2). 
Since $N$ is a finitely generated monoid, $M_{N, \bf S} = M_{\bf S} \cap \bigl ( \mathbb{Z}_{\ge 0} ^d \times N \bigr)$ is also a finitely generated monoid by Theorem \ref{thm:fg}(1). 
Therefore, the intersection $\bigcap _{{\bf S} \in \Lambda} M_{N, \bf S}$ is also a finitely generated monoid again by Theorem \ref{thm:fg}(1). 

For each ${\bf S}  = (S_1, \ldots , S_d)  \in \bigl( 2^{V_{\Gamma} / L} \bigr) ^d$, 
the product  $X_{\bf S} = X_{S_1} \times \cdots \times X_{S_d}$ is a finitely generated $M_{\bf S}$-module since each $X_{S_i}$ is a finitely generated $M_{S_i}$-module by Theorem \ref{thm:NSMN_decomp}(3). 
Since $Y$ is a finitely generated $N$-module, $X_{Y, \bf S} = X_{\bf S} \cap \bigl ( \mathbb{Z}_{\ge 0} ^d \times Y \bigr)$ is also a finitely generated $M_{N, \bf S}$-module by Theorem \ref{thm:fg}(2). 
Therefore, the intersection $\bigcap _{{\bf S} \in \Lambda} X_{Y, \bf S}$ is a finitely generated $\bigl( \bigcap _{{\bf S} \in \Lambda} M_{N, \bf S} \bigr)$-module again by Theorem \ref{thm:fg}(2). 
\end{proof}

\begin{defi}
Let $d$ be a positive integer, and let $Y \subset V_{\Gamma} ^d$ be a subset. 
For ${\bf a} = (a_1, \ldots, a_d) \in \mathbb{Z}_{\ge 0}^d$, we define 
\begin{align*}
B_{Y, \bf a} &:= \bigl \{ (y_1, \ldots, y_d) \in Y  \ \big | \ \text{$d_{\Gamma}(x_0, y_i) \le a_i$ for each $1 \le i \le d$} \bigr\}, \\
S_{Y, \bf a} &:= \bigl \{ (y_1, \ldots, y_d) \in Y \ \big | \ \text{$d_{\Gamma}(x_0, y_i) = a_i$ for each $1 \le i \le d$} \bigr\}, 
\end{align*}
where $x_0$ is the vertex of $\Gamma$ fixed throughout this section. 
Then, we define the \textit{multivariate growth series} $\mathbb{B}_{Y, x_0}({\bf z}),\ \mathbb{S}_{Y, x_0}({\bf z}) \in \mathbb{Z}[[z_1, \ldots , z_d]]$ of $Y$ by
\[
\mathbb{B}_{Y, x_0} ({\bf z}) := \sum _{{\bf a} \in \mathbb{Z}_{\ge 0} ^d} \# B_{Y, \bf a} {\bf z} ^{\bf a}, \qquad
\mathbb{S}_{Y, x_0} ({\bf z}) := \sum _{{\bf a} \in \mathbb{Z}_{\ge 0} ^d} \# S_{Y, \bf a} {\bf z} ^{\bf a}. 
\]
Here, ${\bf z}^{\bf a}$ denotes $z_1 ^{a_1} \cdots z_d ^{a_d}$ for ${\bf a} = (a_1, \ldots , a_d)$. 
\end{defi}

\begin{thm}\label{thm:main_p}
Let $d$ be a positive integer. 
Let $N$ be a finitely generated submonoid of $L^d$. 
Let $Y$ be a finitely generated $N$-submodule of $V_{\Gamma} ^d$. 
Then the multivariate growth series $\mathbb{B}_{Y, x_0}({\bf z})$ and $\mathbb{S}_{Y, x_0}({\bf z})$ of $Y$ are rational functions. 
\end{thm}
\begin{proof}
Since we have $\mathbb{S}_{Y, x_0} ({\bf z}) = (1-z_1) \cdots (1-z_d) \mathbb{B}_{Y, x_0} ({\bf z})$, 
it is sufficient to show that $\mathbb{B}_{Y, x_0} ({\bf z})$ is a rational function. 

For each ${\bf a} \in \mathbb{Z}_{\ge 0} ^d$, 
we have $B_{Y, \bf a} = B_Y \cap \bigl( \{ {\bf a} \} \times  V_{\Gamma} ^d \bigr)$. 
Since we have $B = \bigcup _{S \in 2 ^{V_{\Gamma}/L}} X_S$ by Theorem \ref{thm:NSMN_decomp}(1), we have 
\[
B_Y = \bigcup _{{\bf S} \in ( 2^{V_{\Gamma} / L} ) ^d} X_{Y, {\bf S}}. 
\]
By the inclusion-exclusion principle, in order to show the rationality of 
\[
\mathbb{B}_{Y, x_0}({\bf z}) = 
\sum _{{\bf a} \in \mathbb{Z}_{\ge 0} ^d} \# \Biggl( \Biggl( \bigcup _{{\bf S} \in ( 2^{V_{\Gamma} / L} ) ^d} X_{Y, {\bf S}} \Biggr) \cap \bigl( \{ {\bf a} \} \times  V_{\Gamma} ^d \bigr) \Biggr) {\bf z}^{\bf a} , 
\]
it is sufficient to show the rationality of the series
\[
H _{\Lambda} ({\bf z}) := \sum _{{\bf a} \in \mathbb{Z}_{\ge 0} ^d} \# \left( \left( \bigcap _{{\bf S} \in \Lambda} X_{Y, \bf S} \right) \cap \bigl( \{ {\bf a} \} \times  V_{\Gamma} ^d \bigr) \right) {\bf z} ^{\bf a}
\]
for each subset $\Lambda \subset \bigl( 2^{V_{\Gamma} / L} \bigr) ^d$. 
By Theorem \ref{thm:Serre}, the rationality of each $H _{\Lambda}({\bf z})$ follows from Lemma \ref{lem:fgm}.
\end{proof}

\section{Rationality of the multivariate growth series of algebraic sets of virtually abelian groups}\label{section:proof2}

Following \cite{EL22}, we introduce a concept called ``algebraic sets'' for groups. 
\begin{defi}\label{defi:as}
Let $G$ be a group, and let $d$ be a positive integer. 

\begin{enumerate}
\item 
A \textit{finite system of equations} in $G$ is a finite subset of $G * F_d$, where $F_d$ is the free group generated by $\{ X_1, \ldots , X_d \}$. 

\item
For $w \in G * F_d$ and $(g_1, \ldots , g_d) \in G^d$, we define $w(g_1, \ldots , g_{d}) \in G$ as follows: 
Let $f_1: F_d \to G$ be the group homomorphism satisfying $f_1(X_i) = g_i$ for each $1 \le i \le d$. 
Let $f_2: G * F_d \to G$ be the group homomorphism obtained by extending both $f_1$ and the identity map $G \to G$. 
Then, we define $w(g_1, \ldots , g_{d}) := f_2 (w) \in G$. 

\item
For a finite system $W$ of equations in $G$, 
$(g_1, \ldots , g_d) \in G^d$ is called a \textit{solution} of $W$ when $w(g_1, \ldots, g_d) = 1_G$ holds for all $w \in W$.

\item 
A subset $U \subset G^d$ is called \textit{algebraic} if there exists a finite system $W$ of equations in $G$ such that $U$ is the set of solutions of $W$.  
\end{enumerate}
\end{defi}

\begin{defi}\label{defi:mgs}
Let $G$ be a group, $S \subset G$ be a finite subset, and $\omega$ be a positive integer weight. 
Let $d$ be a positive integer, and let $U \subset G^d$ be a subset. 
\begin{enumerate}
\item We define \textit{multivariate (relative) growth series} $\mathbb{S}_{U, S, \omega}({\bf z})$ of $U$ by
\[
\mathbb{S}_{U, S, \omega}({\bf z}) := \sum _{{\bf x} \in U} z_1^{\omega(x_1)} z_2^{\omega(x_2)} \cdots z_d^{\omega(x_d)} \in \mathbb{Z}[[z_1, \ldots, z_d]]. 
\]
Here, by convention, we define $z_i ^{\omega (x_i)} = 0$ if $\omega (x_i) = \infty$. 

\item We define \textit{univariate (relative) growth series} $\mathbb{S}^{\rm uni}_{U, S, \omega}(t)$ of $U$ by
\[
\mathbb{S}^{\rm uni}_{U, S, \omega}(t) := \sum _{{\bf x} \in U} t^{\omega ({\bf x})} \in \mathbb{Z}[[t]], 
\]
where $\omega ({\bf x}) := \sum _{i = 1} ^{\ell} \omega (x_i)$. 
By definition, we have $\mathbb{S}^{\rm uni}_{U, S, \omega}(t) = \mathbb{S}_{U, S, \omega}(t, \ldots, t)$. 
\end{enumerate}
\end{defi}

In \cite{EL22}*{Corollary 4.16}, Evetts and Levine prove that algebraic subsets of virtually abelian groups are ``coset-wise polyhedral'' (see \cite{EL22}*{Definition 4.12}). 
By using the language of commutative monoids and modules, their result can be described as follows. 

\begin{prop}[\cite{EL22}, \cite{CE}]\label{prop:decomp}
Let $G$ be a finitely generated virtually abelian group with a free abelian normal subgroup $L$ with a finite index. 
Let $U \subset G^d$ be an algebraic subset. 
Then, $U$ can decompose as a finite disjoint union $U = \bigsqcup _{i = 1} ^{\ell} U_i$ so that 
each $U_i$ is a finitely generated $N_i$-module for some finitely generated submonoid $N_i \subset L^d$.  
\end{prop}
\begin{proof}
By Corollary 4.16 in \cite{EL22}, any algebraic set $U \subset G^d$ can decompose as a finite disjoint union $U = \bigsqcup _{j} V_j$ so that each $V_j$ is of the form 
\[
V_j = P_j t^{(j)}
\]
for some $t^{(j)} \in G^d$ and some polyhedral set $P_j \subset L^d$ (see \cite{CE}*{Definition 3.1} for the definition of polyhedral sets). 
Proposition 3.11 and Lemma 3.10 in \cite{CE} show that any polyhedral set $P \subset L^d$ can decompose as a finite disjoint union $P = \bigsqcup _{k} Q_k$ so that each $Q_k$ is of the form 
\[
Q_k = c^{(k)} + \mathbb{Z}_{\ge 0} u^{(k)}_1 + \cdots + \mathbb{Z}_{\ge 0} u^{(k)}_{m^{(k)}} \subset L^d
\]
for some $m^{(k)} \in \mathbb{Z}_{\ge 0}$ and $c^{(k)}, u^{(k)}_1, \ldots, u^{(k)}_{m^{(k)}} \in L^d$. 
Therefore, we conclude that 
$U$ can decompose as a finite disjoint union $U = \bigsqcup _{i} U_i$ so that 
each $U_i$ is of the form 
\[
U_i = N_i t^{(i)},\quad \text{where}\ \ \  N_i = \mathbb{Z}_{\ge 0} u^{(i)}_1 + \cdots + \mathbb{Z}_{\ge 0} u^{(i)}_{m^{(i)}} \subset L^d
\]
for some $t^{(i)} \in G^d$, $m^{(i)} \in \mathbb{Z}_{\ge 0}$ and $u^{(i)}_1, \ldots, u^{(i)}_{m^{(i)}} \in L^d$. 
Here, $N_i$ is the submonoid of $L^d$ generated by finitely many elements $u^{(i)}_1, \ldots, u^{(i)}_{m^{(i)}}$, 
and $U_i$ is the $N_i$-module generated by one element $t^{(i)}$. 
\end{proof}

We prove the rationality of multivariate growth series of algebraic sets of finitely generated virtually abelian groups. 
In \cite{EL22}, Evetts and Levine prove the rationality of the univariate growth series, but the rationality of multivariate growth series remained a conjecture. 

\begin{thm}[cf.\ \cite{EL22}*{Conjecture 5.1}]\label{thm:main}
Let $G$ be a finitely generated virtually abelian group, $S \subset G$ a finite subset, and $\omega$ a positive integer weight. 
Then, for any algebraic subset $U \subset G^d$ of $G$, its multivariate growth series $\mathbb{S}_{Y, S, \omega}({\bf z})$ is a rational function. 
\end{thm}
\begin{proof}
By the same argument as in Remark \ref{rmk:benson}, the Cayley graph $\Gamma = \Gamma _{G,S,\omega}$ of $G$ with respect to $S$ and $\omega$ becomes a periodic graph. 
Therefore, the assertion follows from Proposition \ref{prop:decomp} and Theorem \ref{thm:main_p}. 
\end{proof}

\appendix

\section{Algebraic structure of the graded growth set}\label{section:NSMN}
This appendix is devoted to proving Theorem \ref{thm:NSMN_decomp}. 
The proofs in this section are essentially the same as those in \cite{NSMN21}, but streamlined for mathematicians.
\subsection{Notations and Preliminaries}
Following \cite{IN1}*{Section 2}, we introduce some additional notations on graphs. 
\begin{defi}
Let $\Gamma = (V_{\Gamma}, E_{\Gamma}, s, t, w)$ be a graph. 

\begin{enumerate}
\item
For a walk $p = e_1 e_2 \cdots e_{\ell}$, we define $s(p) := s(e_1)$, $t(p) := t(e_{\ell})$, $\operatorname{length}(p)=\ell$ and 
\[
\operatorname{supp}(p) := \{ s(e_1), t(e_1), t(e_2), \ldots , t(e_{\ell}) \} \subset V_{\Gamma}. 
\]

\item 
A \textit{path} in $\Gamma$ is a walk $e_1e_2 \cdots e_{\ell}$ such that $s(e_1), t(e_1), t(e_2), \ldots, t(e_{\ell})$ are distinct. 
A walk of length $0$ is considered as a path. 

\item
A \textit{cycle} in $\Gamma$ is a walk $e_1 e_2 \cdots e_{\ell}$ with $s(e_1) = t(e_{\ell})$ 
such that $t(e_1), t(e_2), \ldots, t(e_{\ell})$ are distinct. 
A walk of length $0$ is not considered as a cycle. 
$\operatorname{Cyc}_{\Gamma}$ denotes the set of cycles in $\Gamma$. 

\item 
$C_1(\Gamma, \mathbb{Z})$ denotes the group of $1$-chains on $\Gamma$ with coefficients in $\mathbb{Z}$ 
(i.e., $C_1(\Gamma, \mathbb{Z})$ is the free abelian group generated by $E_{\Gamma}$). 
For a walk $p = e_1e_2 \cdots e_{\ell}$ in $\Gamma$, let $\langle p \rangle$ denote the $1$-chain 
$\sum _{i = 1} ^{\ell} e_i \in C_1(\Gamma, \mathbb{Z})$. 
$H_1(\Gamma, \mathbb{Z}) \subset C_1(\Gamma, \mathbb{Z})$ denotes the $1$st homology group 
(i.e., $H_1(\Gamma, \mathbb{Z})$ is the subgroup generated by $\langle p \rangle$ for $p \in \operatorname{Cyc}_{\Gamma}$). 
We refer the reader to \cite{Sunada} for more detail. 
\end{enumerate}
\end{defi}

\begin{defi}
Let $(\Gamma, L)$ be an $n$-dimensional periodic graph. 
\begin{enumerate}
\item Since $L$ is an abelian group, we use the additive notation defined by: 
for $u \in L$ and $x \in V_{\Gamma}$, we define $u+x := u(x)$ as the translation of $x$ by $u$.

\item 
The \textit{quotient graph} $\Gamma /L = (V_{\Gamma/L}, E_{\Gamma/L}, s_{\Gamma/L},t_{\Gamma/L},w_{\Gamma/L})$ is defined by $V_{\Gamma/L} := V_{\Gamma} /L$, $E_{\Gamma/L} := E_{\Gamma} /L$, 
and the functions $s_{\Gamma/L}: E_\Gamma/L \to V_\Gamma/L$, 
$t_{\Gamma/L}: E_\Gamma/L \to V_\Gamma/L$, and 
$w_{\Gamma/L}: E_\Gamma/L \to \mathbb{Z}_{>0}$ induced from $s_{\Gamma}$, $t_{\Gamma}$, and $w_{\Gamma}$.  

\item 
For any $x \in V_{\Gamma}$ and $e \in E_{\Gamma}$, 
let $\overline{x} \in V_{\Gamma /L}$ and $\overline{e} \in E_{\Gamma /L}$ denote their images in 
$V_{\Gamma /L} = V_{\Gamma} /L$ and $E_{\Gamma /L} = E_{\Gamma} /L$. 
For a walk $p = e_1 \cdots e_{\ell}$ in $\Gamma$, 
let $\overline{p} := \overline{e_1} \cdots \overline{e_{\ell}}$ denote its image in $\Gamma /L$.  

\item 
When $x, y \in V_{\Gamma}$ satisfy $\overline{x} = \overline{y}$, there exists an element $u \in L$ such that $u + x = y$. 
Since the $L$-action is free, such $u \in L$ uniquely exists. 
We denote this unique $u$ by $y - x$. 

\item 
Let $q_0$ be a path in $\Gamma /L$, and let $q_1, \ldots, q_{\ell} \in \operatorname{Cyc}_{\Gamma /L}$ be cycles. 
The sequence $(q_0, q_1, \ldots, q_{\ell})$ is called \textit{walkable} 
if there exists a walk $q'$ in $\Gamma / L$ such that 
$\langle q' \rangle = \sum _{i = 0} ^{\ell} \langle q_i \rangle$. 

\item 
Let $\mu : H_1(\Gamma/L, \mathbb{Z}) \to L$ be the group homomorphism defined in \cite{IN1}*{Definition 2.8}. 
This homomorphism is uniquely determined by the following condition (see \cite{IN1}*{Lemma 2.7}): 
For any walk $p$ in $\Gamma$ with $\overline{s(p)} = \overline{t(p)}$, we have $\mu(\langle \overline{p} \rangle) = t(p) - s(p)$. 
\end{enumerate}
\end{defi}

In the proof of Theorem \ref{thm:NSMN_decomp} in Subsection \ref{subsection:a_proof}, 
we will use the following easy lemma on a decomposition of a walk to a path and cycles. 
\begin{lem}[\cite{IN1}*{Lemma 2.12, Remark 2.13(2)}]\label{lem:bunkai}
Let $(\Gamma, L)$ be an $n$-dimensional periodic graph. 
\begin{enumerate}
\item 
For a walk $q'$ in $\Gamma / L$, there exists a walkable sequence $(q_0, q_1, \ldots , q_{\ell})$ such that 
$\langle q' \rangle = \sum _{i = 0} ^{\ell} \langle q_i \rangle$ and $\operatorname{supp}(q') = \bigcup _{i=0}^{\ell} \operatorname{supp}(q_i)$. 

\item 
Let $q_0$ be a path in $\Gamma /L$, and let $q_1, \ldots, q_{\ell} \in \operatorname{Cyc}_{\Gamma /L}$ be cycles. 
Then, $(q_0, q_1, \ldots, q_{\ell})$ is walkable if and only if 
there exists a permutation $\sigma:  \{ 1,2, \ldots, \ell \} \to \{ 1,2, \ldots, \ell \}$ such that
\[
\biggl( \operatorname{supp}(q_0) \cup \bigcup _{1 \le i \le k} \operatorname{supp}(q_{\sigma(i)}) \biggr) 
\cap \operatorname{supp}(q_{\sigma(k+1)}) \not = \emptyset
\] 
holds for any $0 \le k \le \ell -1$.

\item 
For a walk $q'$ in $\Gamma / L$ and a vertex $x_0 \in V_{\Gamma}$ satisfying $s(q') = \overline{x_0}$,
there exists the unique walk $p$ in $\Gamma$ satisfying $\overline{p} = q'$ and $s(p) = x_0$. 
\end{enumerate}
\end{lem}

\subsection{Proof of Theorem \ref{thm:NSMN_decomp}}\label{subsection:a_proof}
In this subsection, we give a proof of Theorem \ref{thm:NSMN_decomp} ($=$ Lemma \ref{lem:easy}${}+{}$Theorem \ref{thm:NSMN_fg}). 
Throughout this subsection, we fix an $n$-dimensional periodic graph $(\Gamma, L)$. 
We also fix $x_0 \in V_{\Gamma}$. 

In Section \ref{section:proof}, we have defined a \textit{graded growth set} $B$ of $\Gamma$ with respect to $x_0$ by
\[
B := \{ (i, y) \in \mathbb{Z}_{\ge 0} \times V_{\Gamma} \mid d_{\Gamma}(x_0, y) \le i \} \subset \mathbb{Z}_{\ge 0} \times V_{\Gamma}. 
\]

\begin{defi}\label{defi:MX}
For each subset $S \subset V_{\Gamma}/L$, we define subsets $M_S \subset \mathbb{Z}_{\ge 0} \times L$ and  $X_S \subset \mathbb{Z}_{\ge 0} \times V_{\Gamma}$ as follows. 
\begin{enumerate}
\item
First, we define $M' _S \subset \mathbb{Z}_{\ge 0} \times L$ by
\[
M'_S 
:= \bigl\{ (w(q), \mu(\langle q \rangle)) \ \big | \ 
\text{$q \in \operatorname{Cyc}_{\Gamma /L}$ such that $\operatorname{supp}(q) \subset S$} \bigr\}. 
\]
Then, we define $M_S$ as the submonoid of $\mathbb{Z}_{\ge 0} \times L$ generated by $M'_S$ and $(1,0)$. 

\item
We define a subset
$X_S \subset \mathbb{Z}_{\ge 0} \times V_{\Gamma}$ by
\[
X_S 
:= 
\left\{ 
(i, y) \in \mathbb{Z}_{\ge 0} \times V_{\Gamma}
\ \middle | \ 
\begin{array}{l}
\text{There exists a walk $p$ in $\Gamma$ from $x_0$ to $y$}\\
\text{such that $w(p) \le i$ and $\operatorname{supp}(\overline{p}) = S$.}
\end{array}
\right\}. 
\]
\end{enumerate}
\end{defi}

\begin{lem}\label{lem:easy}
The following assertions hold. 
\begin{enumerate}
\item
We have $B = \bigcup _{S \in 2 ^{V_{\Gamma}/L}} X_S$. 

\item
For each subset $S \subset V_{\Gamma}/L$, 
the monoid $M_S$ is finitely generated. 
\end{enumerate}
\end{lem}
\begin{proof}
We shall prove (1). 
Suppose $(i,y) \in B$. 
Then, by the definition of $B$, there exists a walk $p$ in $\Gamma$ from $x_0$ to $y$ with $w(p) \le i$. 
Then, we have $(i,y) \in X_S$ for $S = \operatorname{supp}(\overline{p})$. 
Therefore, we have $B \subset \bigcup _{S \in 2 ^{V_{\Gamma}/L}} X_S$. 
Since the opposite inclusion is trivial, the assertion (1) is proved. 

Since $\operatorname{Cyc}_{\Gamma /L}$ is a finite set, $M' _S$ is also a finite set, which proves (2). 
\end{proof}

\begin{thm}[\cite{NSMN21}*{Lemma 2.13}]\label{thm:NSMN_fg}
For each subset $S \subset V_{\Gamma}/L$, $X_S$ is a finitely generated $M_S$-module. 
\end{thm}

\begin{proof}
First, we prove $M_S + X_S \subset X_S$. 
Take $(i,y) \in X_S$. 
Then, by the definition of $X_S$, there exists a walk $p$ in $\Gamma$ from $x_0$ to $y$ satisfying 
$w(p) \le i$ and $\operatorname{supp}(\overline{p}) = S$. 
Take $q \in \operatorname{Cyc}_{\Gamma /L}$ such that $\operatorname{supp}(q) \subset S$. 
Since $\operatorname{supp}(q) \cap \operatorname{supp}(\overline{p}) \not= \emptyset$, 
by Lemma \ref{lem:bunkai}(2)(3), 
there exists a walk $p'$ in $\Gamma$ from $x_0$ such that 
\[
\langle \overline{p'} \rangle 
= \langle q \rangle + \langle \overline{p} \rangle, \qquad
\operatorname{supp}(\overline{p'}) 
= \operatorname{supp}(q) \cup \operatorname{supp}(\overline{p}) = S. 
\]
Then, we have 
\[
t(p') 
= \mu(\langle q \rangle) + t(p)
= \mu(\langle q \rangle) + y. 
\]
Furthermore, we have 
\[
w(p') = w(q) + w(p) \le w(q) + i. 
\]
These show that $\bigl( w(q) + i, \mu(\langle q \rangle) + y \bigr) \in X_S$. 
Hence, we conclude that $M' _S + X_S \subset X_S$. 
Since it is clear that $(1,0) + X_S \subset X_S$, 
we conclude that $X_S$ is an $M_S$-module. 

Next, we show that $X_S$ is generated by 
\[
X'_{S} := \bigl\{ (i,y) \in X_S \ \big| \ i \le W \cdot (\# S)^2 \bigr\}, 
\]
where $W := \max _{e \in E_{\Gamma}} w(e)$. 
Take $(i,y) \in X_S$. 
By the definition of $X_S$, there exists a walk $p$ in $\Gamma$ from $x_0$ to $y$ satisfying 
$w(p) \le i$ and $\operatorname{supp}(\overline{p}) = S$. 
By decomposing $\overline{p}$ (Lemma \ref{lem:bunkai}(1)), 
there exists a walkable sequence $(q_0, q_1, \ldots, q_{\ell})$ such that 
$\langle \overline{p} \rangle = \sum _{i=0} ^{\ell} \langle q_i \rangle$ and 
$\bigcup _{i=0}^{\ell} \operatorname{supp}(q_i) = S$. 
By Lemma \ref{lem:bunkai}(2), by rearranging the indices of $q_1, \ldots, q_{\ell}$, 
we may assume the following condition for each $0 \le j \le \ell -1$: 
\begin{itemize}
\item 
$\left( \bigcup _{0 \le i \le j} \operatorname{supp}(q_i) \right) \cap \operatorname{supp}(q_{j+1}) \not = \emptyset$. 
\end{itemize}
Furthermore, we may also assume the following condition for each $0 \le j \le \ell -1$:
\begin{itemize}
\item
If $\bigcup _{0 \le i \le j} \operatorname{supp}(q_i) \not = S$, then 
$\bigcup _{0 \le i \le j} \operatorname{supp}(q_i) \subsetneq  \bigcup _{0 \le i \le j+1} \operatorname{supp}(q_i)$.
\end{itemize}
In particular, there exists $0 \le \ell' \le \ell$ with the following conditions:  
\begin{itemize}
\item $(q_0, q_1, \ldots, q_{\ell'})$ is a walkable sequence, 
\item $\bigcup _{0 \le i \le \ell'} \operatorname{supp}(q_i) = S$, and  
\item $\ell' \le \# S - \#(\operatorname{supp}(q_0)) = \# S - \operatorname{length}(q_0) - 1$. 
\end{itemize}
Hence, by Lemma \ref{lem:bunkai}(2)(3), there exists 
a path $p'$ in $\Gamma$ from $x_0$ such that 
\begin{align*}
\langle \overline{p'} \rangle 
= \sum _{0 \le i \le \ell'} \langle q_i \rangle , \qquad
\operatorname{supp}(\overline{p'}) 
= \bigcup _{0 \le i \le \ell'} \operatorname{supp}(q_i)
= S. 
\end{align*}
Then, since 
\begin{align*}
\operatorname{length}(p') 
&= \operatorname{length}(q_0) + \sum _{1 \le i \le \ell'} \operatorname{length}(q_i) \\
&\le \operatorname{length}(q_0) + \ell ' \cdot \# S \\
&\le \operatorname{length}(q_0) + (\# S - \operatorname{length}(q_0) - 1) \cdot \# S \\
&\le (\# S)^2, 
\end{align*}
we have 
\[
w(p') \le W \cdot \operatorname{length}(p')  \le W \cdot (\# S)^2. 
\]
Therefore, we conclude $\bigl( w(p'), t(p') \bigr) \in X'_S$. 
Since 
\begin{align*}
&i - w(p') 
= i - w(p) + \sum _{\ell' +1 \le i \le \ell} w(q_i), \\
&y-t(p')=t(p)-t(p') 
= \sum _{\ell' +1 \le i \le \ell} \mu( \langle q_i \rangle), 
\end{align*}
we have 
\[
(i,y) - \bigl( w(p'), t(p') \bigr) 
= \bigl( i-w(p) ,0 \bigr) 
+ \sum _{\ell' +1 \le i \le \ell} \bigl (  w(q_i), \mu( \langle q_i \rangle)  \bigr )
\in M_S. 
\]
Therefore, we have $X_S = M_S + X'_S$. 
Since $X'_S$ is a finite set, we conclude that $X_S$ is a finitely generated $M_S$-module. 
\end{proof}


\begin{bibdiv}
\begin{biblist*}
\bib{AC17}{article}{
   author={Antol\'in, Yago},
   author={Ciobanu, Laura},
   title={Formal conjugacy growth in acylindrically hyperboic groups},
   journal={Int. Math. Res. Not. IMRN},
   date={2017},
   number={1},
   pages={121--157},
%   issn={1073-7928},
%   review={\MR{3632100}},
   doi={\doi{10.1093/imrn/rnw026}},
}

\bib{Ben83}{article}{
   author={Benson, M.},
   title={Growth series of finite extensions of ${\bf Z}^{n}$ are
   rational},
   journal={Invent. Math.},
   volume={73},
   date={1983},
   number={2},
   pages={251--269},
%   issn={0020-9910},
%   review={\MR{714092}},
   doi={\doi{10.1007/BF01394026}},
}

\bib{BG09}{book}{
   author={Bruns, Winfried},
   author={Gubeladze, Joseph},
   title={Polytopes, rings, and $K$-theory},
   series={Springer Monographs in Mathematics},
   publisher={Springer, Dordrecht},
   date={2009},
   doi={\doi{10.1007/b105283}},
}

\bib{Can84}{article}{
   author={Cannon, James W.},
   title={The combinatorial structure of cocompact discrete hyperbolic
   groups},
   journal={Geom. Dedicata},
   volume={16},
   date={1984},
   number={2},
   pages={123--148},
%   issn={0046-5755},
%  review={\MR{0758901}},
   doi={\doi{10.1007/BF00146825}},
}



\bib{CE}{article}{
   author={Ciobanu, Laura},
   author={Evetts, Alex},
   title={Rational sets in virtually abelian groups: languages and growth},
   eprint={arXiv:2205.05621v2},
   doi={\doi{10.48550/arXiv.2205.05621}},
}

\bib{CHHR16}{article}{
   author={Ciobanu, Laura},
   author={Hermiller, Susan},
   author={Holt, Derek},
   author={Rees, Sarah},
   title={Conjugacy languages in groups},
   journal={Israel J. Math.},
   volume={211},
   date={2016},
   number={1},
   pages={311--347},
%   issn={0021-2172},
%   review={\MR{3474966}},
   doi={\doi{10.1007/s11856-015-1274-5}},
}

\bib{DO15}{article}{
   author={Davis, Tara C.},
   author={Olshanskii, Alexander Yu.},
   title={Relative subgroup growth and subgroup distortion},
   journal={Groups Geom. Dyn.},
   volume={9},
   date={2015},
   number={1},
   pages={237--273},
%   issn={1661-7207},
%   review={\MR{3343353}},
   doi={\doi{10.4171/GGD/312}},
}

\bib{DS19}{article}{
   author={Duchin, Moon},
   author={Shapiro, Michael},
   title={The Heisenberg group is pan-rational},
   journal={Adv. Math.},
   volume={346},
   date={2019},
   pages={219--263},
%   issn={0001-8708},
%  review={\MR{3908251}},
   doi={\doi{10.1016/j.aim.2019.01.046}},
}

\bib{Eve19}{article}{
   author={Evetts, Alex},
   title={Rational growth in virtually abelian groups},
   journal={Illinois J. Math.},
   volume={63},
   date={2019},
   number={4},
   pages={513--549},
%   issn={0019-2082},
%   review={\MR{4032813}},
   doi={\doi{10.1215/00192082-8011497}},
}

\bib{EL22}{article}{
   author={Evetts, Alex},
   author={Levine, Alex},
   title={Equations in virtually abelian groups: languages and growth},
   journal={Internat. J. Algebra Comput.},
   volume={32},
   date={2022},
   number={3},
   pages={411--442},
   issn={0218-1967},
%   review={\MR{4417480}},
  doi={\doi{10.1142/S0218196722500205}},
}		


\bib{GS19}{article}{
   author={Goodman-Strauss, C.},
   author={Sloane, N. J. A.},
   title={A coloring-book approach to finding coordination sequences},
   journal={Acta Crystallogr. Sect. A},
   volume={75},
   date={2019},
   number={1},
   pages={121--134},
   doi={\doi{10.1107/S2053273318014481}},
}

\bib{Gro87}{article}{
   author={Gromov, M.},
   title={Hyperbolic groups},
   conference={
      title={Essays in group theory},
   },
   book={
      series={Math. Sci. Res. Inst. Publ.},
      volume={8},
      publisher={Springer, New York},
   },
%   isbn={0-387-96618-8},
   date={1987},
   pages={75--263},
%   review={\MR{0919829}},
   doi={\doi{10.1007/978-1-4613-9586-7\_3}},
}

\bib{GKBS96}{article}{
  title={Algebraic description of coordination sequences and exact topological densities for zeolites},
  author={Grosse-Kunstleve, R. W},
  author={Brunner, G. O},
  author={Sloane, N. J. A},
  journal={Acta Crystallogr. Sect. A},
  volume={52},
  number={6},
  pages={879--889},
  year={1996},
  doi={\doi{10.1107/S0108767396007519}},
}

\bib{GS10}{article}{
   author={Guba, Victor},
   author={Sapir, Mark},
   title={On the conjugacy growth functions of groups},
   journal={Illinois J. Math.},
   volume={54},
   date={2010},
   number={1},
   pages={301--313},
%   issn={0019-2082},
%   review={\MR{2776997}},
}


\bib{IN1}{article}{
   author={Inoue, Takuya},
   author={Nakamura, Yusuke},
   title={Ehrhart theory on periodic graphs},
   journal={Algebr. Comb.},
   volume={7},
   date={2024},
   number={4},
   pages={969--1010},
   doi={\doi{10.5802/alco.367}},
%   review={\MR{4804581}},
}

\bib{IN2}{article}{
   author={Inoue, Takuya},
   author={Nakamura, Yusuke},
   title={Ehrhart theory on periodic graphs II: Stratified Ehrhart ring theory},
   eprint={arXiv:2310.19569v2}, 
  doi={\doi{10.48550/arXiv.2310.19569}},
}

\bib{NSMN21}{article}{
  author={Nakamura, Yusuke},
  author={Sakamoto, Ryotaro},
  author={Mase, Takafumi},
  author={Nakagawa, Junichi},
  title={Coordination sequences of crystals are of quasi-polynomial type},
  journal={Acta Crystallogr. Sect. A},
  volume={77},
  number={2},
  pages={138--148},
  year={2021},
  doi={\doi{10.1107/S2053273320016769}},
}

\bib{Ogu}{book}{
   author={Ogus, Arthur},
   title={Lectures on Logarithmic Algebraic Geometry},
   series={Cambridge Studies in Advanced Mathematics},
   volume = {178},
   publisher={Cambridge University Press, Cambridge},
   date={2018}, 
}  


\bib{Sunada}{book}{
   author={Sunada, Toshikazu},
   title={Topological crystallography},
   series={Surveys and Tutorials in the Applied Mathematical Sciences},
   volume={6},
   note={With a view towards discrete geometric analysis},
   publisher={Springer, Tokyo},
   date={2013},
%   pages={xii+229},
%   isbn={978-4-431-54176-9},
%   isbn={978-4-431-54177-6},
%   review={\MR{3014418}},
   doi={\doi{10.1007/978-4-431-54177-6}},
}
\end{biblist*}
\end{bibdiv}
\end{document}